\documentclass[final,nomarks]{dmtcs-episciences}

\usepackage{amsmath}
\usepackage{amssymb,latexsym}
\usepackage{amsthm}
\usepackage{tikz}

\usepackage{palatino}
\usepackage{eulervm}

\makeatletter
\newtheorem*{rep@theorem}{\rep@title}
\newcommand{\newreptheorem}[2]{
\newenvironment{rep#1}[1]{
\def\rep@title{#2~\ref{##1}}
\begin{rep@theorem}}
{\end{rep@theorem}}}
\makeatother

\newreptheorem{thm}{Theorem}
\newtheorem*{thm*}{Theorem}

\newreptheorem{cor}{Corollary}

\newreptheorem{lemma}{Lemma}

\newtheorem*{conj*}{Conjecture}

\newtheorem{thmO}{Theorem}
\newtheorem{propO}[thmO]{Proposition}
\newtheorem{corO}[thmO]{Corollary}

\newtheorem{questO}[thmO]{Question}

\theoremstyle{definition}

\newtheorem{defn*}{Definition}

\newtheorem*{example*}{Example}

\newtheorem*{comment*}{Comment}

\newenvironment{bullets} {\vspace{-3pt}\begin{itemize}\itemsep0pt} {\end{itemize}\vspace{-3pt}}

\newcommand{\CCC}{\mathcal{C}}
\newcommand{\DDD}{\mathcal{D}}

\newcommand{\GGG}{\mathcal{G}}
\newcommand{\HHH}{\mathcal{H}}

\newcommand{\KKK}{\mathcal{K}}
\newcommand{\LLL}{\mathcal{L}}

\newcommand{\PPP}{\mathcal{P}}
\newcommand{\QQQ}{\mathcal{Q}}

\newcommand{\SSS}{\mathcal{S}}
\newcommand{\TTT}{\mathcal{T}}
\newcommand{\UUU}{\mathcal{U}}

\newcommand{\ZZZ}{\mathcal{Z}}

\newcommand{\+}{\hspace{0.07 em}}

\newcommand{\half}{\frac{1}{2}}
\newcommand{\liminfty}[1][n]{\lim\limits_{#1\rightarrow\infty}}
\newcommand{\seq}[1]{\text{\textsc{Seq}}[#1]}
\newcommand{\seqbig}[1]{\text{\textsc{Seq}}\big[#1\big]}
\newcommand{\seqplus}[1]{\text{\textsc{Seq}}^+[#1]}
\newcommand{\seqfrac}[2][1]{\frac{#1}{1-#2}}
\newcommand{\seqplusfrac}[2][]{\frac{#1#2}{1-#2}}
\newcommand{\av}{\mathsf{Av}}
\newcommand{\gr}{\mathrm{gr}}
\newcommand{\Grid}{\mathsf{Grid}}

\newcommand{\plotptradius}{0.275}
\newcommand{\plotpt}[3][] 
{ \fill[#1,radius=\plotptradius] (#2,#3) circle; }

\newcommand{\plotperm}[3][]  
{
  \foreach \y [count=\x] in {#3} \plotpt[#1]{\x}{\y}
  \draw[thick] (.5,.5) rectangle (#2.5,#2.5);
}

\newcommand{\classD}{\av(\mathbf{4213},\mathbf{2143})}
\newcommand{\classH}{\av(\mathbf{4213},\mathbf{2413},\mathbf{2143})}

\newcommand{\tree}[1]{\text{\textsc{Tree}}[#1]}
\newcommand{\treebig}[1]{\text{\textsc{Tree}}\big[#1\big]}

\newcommand{\GenGCTwoRaw}[2]{
\scalebox{0.9}{$
\begin{array}{|c|}
\hline
\!\!#1\!\! \\
\hline
\!\!#2\!\! \\
\hline
\end{array}
$}
}
\newcommand{\GenGCTwo}[2]{
$
\GenGCTwoRaw{#1}{#2}
$
}
\newcommand{\GenGCTwoAv}[2]{\GenGCTwo{\av(\mathbf{#1})}{\av(\mathbf{#2})}}
\newcommand{\GenGCTwoAvRaw}[2]{\GenGCTwoRaw{\av(\mathbf{#1})}{\av(\mathbf{#2})}}

\newcommand{\GenGCThreeRaw}[3]{
\scalebox{0.9}{$
\begin{array}{|c|c|}
\hline
& \!\!#1\!\! \\
\hline
\!\!#2\!\! & \!\!#3\!\! \\
\hline
\end{array}
$}
}
\newcommand{\GenGCThree}[3]{
$
\GenGCThreeRaw{#1}{#2}{#3}
$
}
\newcommand{\GenGCThreeAv}[3]{\GenGCThree{\av(\mathbf{#1})}{\av(\mathbf{#2})}{\av(\mathbf{#3})}}
\newcommand{\GenGCThreeAvRaw}[3]{\GenGCThreeRaw{\av(\mathbf{#1})}{\av(\mathbf{#2})}{\av(\mathbf{#3})}}

\author{David Bevan\affiliationmark{1}}
\title{The permutation class Av(4213,2143)}
\affiliation{University of Strathclyde, Glasgow, Scotland}
\received{2015-10-22}
\accepted{2016-10-13}
\begin{document}
\publicationdetails{18}{2017}{2}{10}{1309}
\maketitle

\begin{abstract}
We determine
the structure of permutations avoiding the patterns $\mathbf{4213}$ and $\mathbf{2143}$.
Each such permutation consists of the skew sum of a sequence of plane trees, together with an increasing sequence of points above and an increasing sequence of points to its left.
We use this characterisation to establish the generating function enumerating these permutations.
We also investigate the properties of a typical large permutation in the class
and prove that if a large permutation
that avoids $\mathbf{4213}$ and $\mathbf{2143}$
is chosen uniformly at random, then it is more likely than not to avoid $\mathbf{2413}$ as well.
\end{abstract}

\section{Introduction}

Towards the end of his talk at \emph{Permutation Patterns 2015}, Vince Vatter pointed out that four of the permutation classes
avoiding two patterns of length four
remaining to be enumerated
are subclasses of \emph{generalised grid classes}.
We illustrate the value of this observation by using it to
determine the generating function of the class
$\DDD=\classD$ and to explore the structure of permutations in $\DDD$.

Breaking somewhat with tradition, we assume familiarity with the standard terminology and notation used for permutation classes. See~\cite{BevanPPBasics} for a very brief introduction, or Vatter's survey article~\cite{VatterSurvey} for a much more thorough exposition.
Other overviews of the field can be found in the books by
B\'ona~\cite{Bona2012} and
Kitaev~\cite{Kitaev2011}.

The concept of a generalised grid class was introduced by Vatter in~\cite[Section~2]{Vatter2011}.
A generalised grid class is defined by a matrix $M=(M_{i,j})$, whose entries are permutation classes.
The class $\Grid(M)$ consists of those permutations $\sigma$
whose plot can be partitioned into rectangular cells, one corresponding to each entry of $M$,
such that the subpermutation of $\sigma$ consisting of any points in the cell corresponding to $M_{i,j}$ is a member of the permutation class $M_{i,j}$.
Such a partition is known as a \emph{gridding} of $\sigma$. See Figure~\ref{figGridClassEx} for an illustration. With a slight abuse of notation, we tend to identify a grid class with its defining matrix.
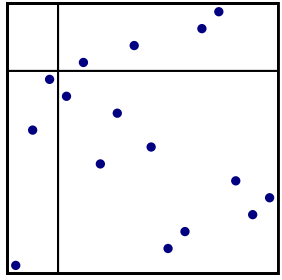
\begin{figure}[ht]
$$
\begin{tikzpicture}[scale=0.225]
\plotperm[blue!50!black]{16}{1,9,12,11,13,7,10,14,8,2,3,15,16,6,4,5}
\draw[thick] (0.5,12.5)--(16.5,12.5);
\draw[thick] (3.5,0.5)--(3.5,16.5);
\end{tikzpicture}
\vspace{-6pt}
$$
\caption{The gridding of a permutation in generalised grid class \protect\GenGCThreeAv{21}{21}{213}}
\label{figGridClassEx}
\end{figure}

We prove (Proposition~\ref{propAinG}) that
$\DDD=\classD$
is a subclass of
the generalised grid class $$\GenGCThreeAvRaw{21}{21}{213}\,,$$
each permutation in $\DDD$ consisting of the skew sum of a sequence of plane trees, with an increasing sequence of points above (\emph{top points})
and an increasing sequence of points to its left (\emph{left points}).
See Figures~\ref{figGridClassEx} and~\ref{figDExample} for illustrations.
By analysing the structure of certain ``canonical'' griddings, we establish (Theorem~\ref{thmAgf}) that
the generating function enumerating $\DDD$ is
$$
D(z) \;=\;
\frac{(1-2\+z) \big({-}1+5\+z-7\+z^2+2\+z^3 +(1-z) \sqrt{1-6\+z+5\+z^2}\big)}{1-10\+z+24\+z^2-20\+z^3+4\+z^4} .
$$
We also investigate the properties of a typical large permutation in $\DDD$ and determine (Corollary~\ref{corD}) that,
asymptotically, one fifth of the points in a permutation in $\DDD$ are top points, and that
more than half of the permutations in $\DDD$ have a gridding with no left points and hence contain no occurrences of $\mathbf{2413}$.

In the rest of this introductory section, we present various definitions, notation and results that we use later.
In Section~\ref{sectH}, we investigate the structure of $\classH$, a subclass of $\DDD$ first studied by Bloom \& Burstein~\cite{BB2014}.
Then, in Section~\ref{sectD}, we build on our analysis of this subclass to
explore the structure of permutations in $\DDD$ and
prove 
our main results.
We conclude with some questions.

To establish generating functions, we make use of the \emph{symbolic method} as presented in Flajolet \& Sedge\-wick~\cite[Chapters I and III]{FS2009}.
We use $\ZZZ$ to denote the atomic class consisting of a single point in a permutation,
$\seq{\CCC}$ to represent a possibly empty sequence of elements of $\CCC$ and
$\seqplus{\CCC}$ to represent a non-empty sequence of elements of $\CCC$.

We also use $\tree{\CCC}$ to denote the class of non-empty plane trees whose vertices are drawn from $\CCC$, defined by the structural equations
$$
\tree{\CCC} \;=\; \CCC \times \seqbig{\tree{\CCC}} \qquad{\text or}\qquad \tree{\CCC} \;=\; \CCC + \tree{\CCC}^2.
$$
We make repeated use of the fact that any permutation $\tau\in\av(\mathbf{213})$ consists of the skew sum of a sequence of plane trees, and thus $\av(\mathbf{213}) \cong \seq{\tree{\ZZZ}}$.
This is made evident by constructing the \emph{Hasse graph} of $\tau$, the graph corresponding to the Hasse diagram of the poset on the points of $\tau$ in which $u<v$ if $u$ is to the lower left of $v$.
See the lower cell in Figure~\ref{figHExample} and the lower right cell in Figure~\ref{figDExample} for illustrations.
Hasse graphs of permutations were introduced by 
Bousquet-M\'elou \& Butler in~\cite{BMB2007}, and previously used by the present author in~\cite{BevanAv1324GR} and~\cite{BevanTwoClasses}.

We make use of the following terminology.
The uppermost (and leftmost) tree in such a skew sum is the \emph{upper tree}.
In each tree, the left to right maxima (the vertices in the leftmost maximal path from the root) constitute the \emph{trunk}, the uppermost point of which is its \emph{tip}.
The \emph{uppermost branching point} in the trunk of a tree is the uppermost point in the trunk of degree greater than~2.

To conclude this introductory section, we state three general propositions of analytic combinatorics that we require.
Recall that $\big[z^n\big]f(z)$ denotes the coefficient of $z^n$ in $f(z)$.

Proposition~\ref{propAsympt} enables us to determine the complete asymptotics of $\big[z^n\big]f(z)$ by repeated application
to successive terms of the (Puiseux) expansion for $f(z)$ around its dominant singularity, and Proposition~\ref{propMoments} can be used to determine the mean and standard deviation of some parameter.

\begin{propO}[{Flajolet \& Odlyzko~\cite{FO1990}; see~\cite[Theorem~VI.1]{FS2009}}]\label{propAsympt}
  The coefficient of $z^n$ in $\lambda\+(1-z/\rho)^\alpha$ admits for large $n$ the following asymptotic expansion:
$$
\big[z^n\big]\lambda\+(1-z/\rho)^\alpha \;\sim\; \frac{\lambda}{\Gamma(-\alpha)} \+ \rho^n \+ n^{-\alpha-1} \+ \bigg(\!1 \,+\, \sum_{k=1}^\infty \frac{e_k}{n^k}\bigg),
$$
where
$$
e_k \;=\; \sum_{\ell=k}^{2k} \+\lambda_{k,\ell} \prod_{j=1}^\ell (\alpha+j)
\qquad
\text{and}
\qquad
\lambda_{k,\ell} \;=\; \big[v^k t^\ell\big]e^t\+(1+v\+t)^{-1-1/v}.
$$
\end{propO}

\begin{propO}[{\cite[Proposition III.2]{FS2009}}]\label{propMoments}
Suppose that $A(z,x)$ is the bivariate generating function for some combinatorial class, in which $z$ marks size and $x$ marks the value of a parameter $\xi$.
Then the mean
and
variance
of $\xi$ for elements of size $n$ are given by
$$
\mathbb{E}_n[\xi] \;=\; \frac{\big[z^n\big]\+ \partial_x A(z,x)\big|_{x=1}}{\big[z^n\big]\+ A(z,1)}
\qquad
\text{and}
\qquad
\mathbb{V}_n[\xi]
\;=\;
\frac{\big[z^n\big]\+ \partial^2_x A(z,x)\big|_{x=1}}{\big[z^n\big]\+ A(z,1)}
\:+\: \mathbb{E}_n(\xi)
\:-\: {\mathbb{E}_n(\xi)}^2
$$
respectively.
\end{propO}

Proposition~\ref{propGaussian}, which is a consequence of Hwang's Quasi-powers Theorem~\cite{Hwang1998}, allows us to establish that a parameter is asymptotically normally distributed from the fact that its {\small PGF} behaves nearly like a large power of a fixed function.

\begin{propO}[{\cite[Theorem~IX.12]{FS2009}}]\label{propGaussian}
Let $F(z,x)$ be a bivariate function, analytic at $(0,0)$ and with non-negative Taylor coefficients, and
let $\xi_n$ be the sequence of random variables with probability generating functions
$$
\frac{\big[z^n\big]F(z,x)}{\big[z^n\big]F(z,1)} .
$$
If the following conditions hold:
\begin{bullets}
\item[(i)]
There exist positive $r$ and $\varepsilon$ and functions $A$, $B$ and $C$, analytic in $\big\{(z,x):|z|\leqslant r,|x-1|\leqslant\varepsilon\big\}$, such that
$F(z,x)=A(z,x)+B(z,x)C(z,x)^\alpha$ for some $\alpha\notin\mathbb{N}_0$.
\item[(ii)] The equation $C(z,1)=0$ has a unique root, $\rho$, with $|\rho|\leqslant r$, $\rho$ is a simple root, $B(\rho,1)\neq0$, $\partial_z C(\rho,1)\neq0$  and $\partial_x C(\rho,1)\neq0$.
\end{bullets}
Then, as long as its asymptotic variance is non-zero, $\xi_n$
converges in law to a Gaussian distribution.
\end{propO}

\section{Av(4213,2413,2143)}\label{sectH}

Let us begin by considering a subclass of $\DDD$,
the permutation class $\classH$.
First, we present a characterisation of the class, which we then use to establish its generating function and determine some asymptotic properties.

This class was first investigated 
by Bloom \& Burstein~\cite{BB2014}.
Although they don't explicitly make use of the notion of a generalised grid class, they prove a result equivalent to our first proposition.

\begin{propO}[{\cite[Lemma~1.5]{BB2014}}]\label{propH}
  $\classH$ is
  the generalised grid class $\HHH=\,\GenGCTwoAv{21}{213}\!\!$.
\end{propO}

\begin{figure}[ht]
\vspace{6pt}
$$
\begin{tikzpicture}[scale=0.25]
\plotperm[blue!50!black]{4}{4,2,1,3}
\draw[gray] (0.5,3.5)--(4.5,3.5);
\node[] at (2.5,-0.5) {$\mathbf{4213}$};
\end{tikzpicture}
\qquad\quad\qquad
\begin{tikzpicture}[scale=0.25]
\plotperm[blue!50!black]{4}{2,4,1,3}
\draw[gray] (0.5,3.5)--(4.5,3.5);
\node[] at (2.5,-0.5) {$\mathbf{2413}$};
\end{tikzpicture}
\qquad\quad\qquad
\begin{tikzpicture}[scale=0.25]
\plotperm[blue!50!black]{4}{2,1,4,3}
\draw[gray] (0.5,3.5)--(4.5,3.5);
\node[] at (2.5,-0.5) {$\mathbf{2143}$};
\end{tikzpicture}
\vspace{-9pt}
$$
\caption{The basis of class $\HHH$}
\label{figHBasis}
\end{figure}

\begin{proof}
Let $B=\{\mathbf{4213},\mathbf{2413},\mathbf{2143}\}$.
Suppose we have a gridding of a permutation $\pi\in\HHH$,
and that $\beta\in B$.
Inspection of Figure~\ref{figHBasis} shows that no more than one point of an occurrence of $\beta$ in $\pi$ can occur in the upper cell, or else it would contain a $\mathbf{21}$, and
neither can three or more points of $\beta$ occur in the lower cell, or else it would contain a $\mathbf{213}$.
Thus $\HHH$ is a subclass of $\av(B)$.

On the other hand, suppose $\sigma$ is a permutation in $\av(B)$ of length $n$, and let $\tau$ be the maximal increasing subsequence
of $\sigma$ of the form $(k,k+1,\ldots,n-1,n)$ for some $k\leqslant n$.
Note that if $k>1$, then, by the maximality of $\tau$, the term $k-1$ occurs to the right of $k$ in $\sigma$.

We claim that $\sigma$
can be gridded in $\HHH$ by placing the points of $\tau$ in the upper cell, and any remaining points in the lower cell.
Clearly, $\tau$ avoids $\mathbf{21}$.
Suppose there was an occurrence of $\mathbf{213}$ in the lower cell. Inspection of Figure~\ref{figHBasis} shows that, since $\sigma$ avoids elements of $B$, the occurrence of $\mathbf{213}$ must occur to the left of all the points in the upper cell. However, in this case, the leftmost two points of this $\mathbf{213}$ together with $k$ and $k-1$ would form a $\mathbf{2143}$, which is impossible because $\sigma$ avoids $\mathbf{2143}$.
Thus $\av(B)$ is a subclass of $\HHH$.
\end{proof}

Proposition~\ref{propH} can also be deduced from Atkinson~\cite[Theorem~2.2]{Atkinson1999},
which specifies the structure of basis elements of juxtapositions of finitely based permutation classes.

Bloom \& Burstein also enumerate $\HHH$. We present our own proof here so that we can make use of it later when considering class $\DDD$.
\begin{propO}[{\cite[Lemma~1.6]{BB2014}}]\label{propHgf}
The generating function for $\HHH=\classH$ is
$$
H(z) \;=\;
\frac{1-3\+ z+2\+ z^2-\sqrt{1-6\+ z+5\+ z^2}}{2\+ z\+(2-z) } . 
$$
\end{propO}

\begin{figure}[ht]
$$
\begin{tikzpicture}[scale=0.25]
\path [fill=blue!15] (6.8,0.5) rectangle (7.2,21.5);
\path [fill=blue!15] (10.8,0.5) rectangle (11.2,21.5);
\path [fill=blue!15] (11.8,0.5) rectangle (12.2,21.5);
\path [fill=blue!15] (15.8,0.5) rectangle (16.2,21.5);
\path [fill=blue!15] (17.8,0.5) rectangle (18.2,21.5);
\path [fill=blue!15] (20.8,0.5) rectangle (21.2,21.5);
\draw[thick,dashed] (13.5,0.5)--(13.5,21.5);
\draw[thick,dashed] (16.5,0.5)--(16.5,21.5);
\draw[thick] (20,2)--(17,1)--(19,3);
\draw[thick] (14,4)--(15,5);
\draw[thick] (1,6)--(2,9)--(3,10)--(4,14)--(5,15);
\draw[thick] (10,8)--(1,6)--(13,7);
\draw[thick] (9,11)--(3,10)--(6,12)--(8,13);
\plotperm[blue!50!black]{21}{6,9,10,14,15,12,16,13,11,8,17,18,7,4,5,19,1,20,3,2,21}
\draw[thick] (0.5,15.5)--(21.5,15.5);
\end{tikzpicture}
\vspace{-6pt}
$$
\caption{The canonical gridding of a permutation in $\HHH$}
\label{figHExample}
\end{figure}
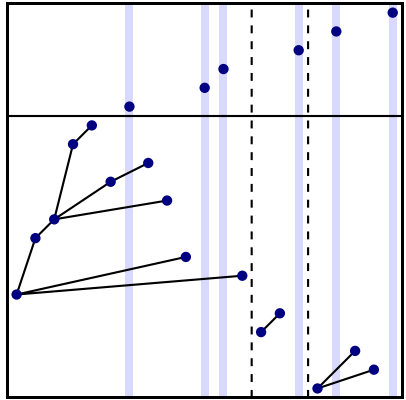

\newcommand{\TsupT}{\TTT}
\newcommand{\tsupT}{T}
\newcommand{\Ttop}{\UUU}
\newcommand{\ttop}{U}
\newcommand{\Pleft}{\LLL}
\newcommand{\pleft}{L}
\newcommand{\Psplit}{\QQQ}
\newcommand{\psplit}{Q}
\newcommand{\Tsplit}{\SSS}
\newcommand{\tsplit}{S}

\begin{proof}
For each permutation $\sigma$ in $\HHH$, we consider the gridding of $\sigma$ in which the number of points in the upper cell is {minimal}
(see Figure~\ref{figHExample} for an illustration).
We call this the \emph{canonical} gridding of $\sigma$ in $\HHH$, and refer to the points in the top cell as the \emph{top points}.
Clearly, there is a bijection between permutations in $\HHH$ and their canonical griddings.

The minimality of the set of top points implies that they all occur to the right of the first occurrence of a $\mathbf{21}$ in the lower cell.
In other words, the leftmost top point must be to the right of the (tip of the) trunk of the upper tree and not adjacent to it.

We use the symbolic method to construct the class, marking top points with $t$ for clarity (although we don't actually need an extra catalytic variable). The structural equations are as follows:
\begin{align}
  \PPP & \;=\;  \seqplus{\ZZZ}  \label{eqTpath} \\[5pt]
  \TsupT & \;=\;  \treebig{\ZZZ \times \seq{t\+\ZZZ}}  \label{eqTr} \\[5pt]
  \Ttop & \;=\;  \PPP \:\;+\;\: \seqplus{\ZZZ} \times (\ZZZ \times \seqplus{\TsupT}) \times \seqbig{\ZZZ \times \seq{\TsupT}}  \label{eqTtop} \\[5pt]
  \HHH & \;=\;  \Ttop \:\times\: \seq{\TsupT}  \label{eqH}
\end{align}
     In~\eqref{eqTpath},
     we use $\PPP$ to denote
     a tree consisting only of a trunk (a \emph{path} tree).

     In~\eqref{eqTr},
     we use $\TsupT$ to denote
     a tree 
     together with any top points that occur to the right of some of its vertices.
     In Figure~\ref{figHExample}, the dashed lines illustrate how the top points are associated with the trees.

In~\eqref{eqTtop},
      we use $\Ttop$ to denote
      the upper tree, together with its top points. This is either a path tree (with no top points), or else is constructed (tip downwards) from 
      \begin{bullets}
        \item[(i)]   at least one point on the trunk above the uppermost branching point,
        \item[(ii)]  the uppermost branching point on the trunk, rooting at least one subtree, which may have top points to the right of some of its vertices, and
        \item[(iii)] any points on the trunk below the uppermost branching point, each possibly rooting a subtree that may have top points to the right of some of its vertices.
      \end{bullets}

Finally, in~\eqref{eqH}, the class $\HHH$ is constructed from the upper tree, $\Ttop$, and zero or more $\TsupT$ trees.

Translation into functional equations yields:
\begin{align*}
  \tsupT(z,t)   & \;=\;  \frac{z}{(1-t\+z)(1-\tsupT(z,t))} ,
  \\[6pt]
  \ttop(z,t) & \;=\;  \seqplusfrac{z} \:\;+\;\: \seqplusfrac{z} \times \seqplusfrac[z\+]{\tsupT(z,t)} \times \seqfrac{\seqfrac[z]{\tsupT(z,t)}} ,
  \\[6pt]
  H(z,t)     & \;=\;   \seqfrac[\ttop(z,t)]{\tsupT(z,t)} .
\end{align*}
Solving, expanding, simplifying and rearranging gives
$$
H(z,t) \;=\;
\frac{1 - (t+2)\+z + 2\+t\+z^2 - \sqrt{1 - (2\+t+4)\+z + t\+(t+4)\+z^2}}{2\+z\+(t+1-t\+z) } .
$$
The univariate generating function then results from setting $t=1$.
\end{proof}

From the generating function, we can determine the following asymptotic behaviour.
\begin{corO}\label{corH}
Asymptotically, $|\HHH_n| \sim \frac{5}{18}\sqrt{\frac{5}{\pi}}\+5^n\+n^{-3/2}$ and the 
number of top points in the canonical gridding of a permutation in $\HHH_n$
is normally distributed with mean $\frac{n}{5}$ and standard deviation $\frac{2\sqrt{n}}{5}$.
\end{corO}

\begin{proof}
These results follow in the standard way from 
singularity analysis and the extraction of moments from the generating function.

The asymptotics of $|\HHH_n|$ results from the application of Proposition~\ref{propAsympt} to the first nonconstant term of
the Puiseux expansion for $H(z)$ around its dominant singularity $\rho=\frac{1}{5}$:
$$
H(z) \;\sim\; \tfrac{2}{3} \:-\: \tfrac{5\sqrt{5}}{9}\sqrt{1-5\+z} \:+\: O(1-5\+z).
$$

The asymptotic mean and standard deviation for the number of top points follow by the application of Propositions~\ref{propAsympt} and~\ref{propMoments}.
It is also readily verified that $H(z,t)$ is amenable to Proposition~\ref{propGaussian}. Indeed, in the notation of that proposition, we have
\begin{align*}
  A(z,t) & \;=\; \frac{1 - (t+2)\+z + 2\+t\+z^2}{2\+z\+(t+1-t\+z) } , \\[5pt]
  B(z,t) & \;=\; \frac{-1}{2\+z\+(t+1-t\+z) } , \\[5pt]
  C(z,t) & \;=\; 1 - (2\+t+4)\+z + t\+(t+4)\+z^2 ,
\end{align*}
and the exponent $\alpha=1/2$. The relevant conditions are then satisfied for any $r\in[1/5,1)$ and $\varepsilon\in(0,1)$, with
$\rho=1/5$, as expected, being the unique root of $C(z,1)=0$ whose magnitude is at most $r$. The number of top points thus 
exhibits a Gaussian limit distribution concentrated around $n/5$.
\end{proof}

The first twelve terms of the sequence $|\HHH_n|$ are 1, 2, 6, 21, 79, 311, 1265, 5275, 22431, 96900, 424068, 1876143.
More values
can be found at
\href{http://oeis.org/A033321}{A033321} in OEIS~\cite{OEIS}.

\section{Av(4213,2413)}\label{sectD}

Let us now turn to a consideration of the structure and enumeration of class $\DDD$.
In the first two propositions, we present a characterisation of the class, which we then use to enumerate the class in our main theorem.
We conclude by determining some asymptotic properties.

\begin{figure}[t]
$$
\boxed{
\begin{array}{c}
\av(\mathbf{4213})                                                  \\[4pt]
\bigcup                                                             \\[6pt]
\GGG \;=\; \,\GenGCThreeAvRaw{21}{21}{213} \phantom{ \;=\; \GGG}    \\[12pt]
\bigcup                                                             \\[4pt]
\DDD \;=\; \classD \phantom{ \;=\; \DDD}                            \\[4pt]
\bigcup                                                             \\[4pt]
\hspace{-.35in}
\phantom{\,\GenGCTwoAvRaw{21}{213}\, \;=\; }
\HHH \;=\; \classH \;=\; \,\GenGCTwoAvRaw{21}{213}
\phantom{ \;=\; \HHH}                                               
\hspace{-.35in}
\end{array}
}
$$
\centering{\small$\av(\mathbf{4213})$ is enumerated by \href{http://oeis.org/A022558}{A022558}, $\DDD$ by \href{http://oeis.org/A165537}{A165537}, and $\HHH$ by \href{http://oeis.org/A033321}{A033321}.\\[4pt]}
\caption{The permutation class hierarchy}
\label{figClassInclusions}
\end{figure}

\begin{propO}\label{propAinG}
The 
class
$\DDD=\classD$ is a subclass of
$\GGG=\,\GenGCThreeAv{21}{21}{213}\!\!$, which is itself a subclass of $\av(\mathbf{4213})$.
\end{propO}
These class inclusions are illustrated in Figure~\ref{figClassInclusions}.

\begin{proof}
Suppose $\sigma$ is a permutation in $\DDD$ of length $n$, and
let $\tau$ be the maximal increasing subsequence of $\sigma$ of the form $(k,k+1,\ldots,n-1,n)$ for some $k\leqslant n$.
If $k>1$, then, by the maximality of $\tau$, the term $k-1$ occurs to the right of $k$ in $\sigma$.
Now let $\lambda$ be the maximal increasing prefix $(\sigma_1,\ldots,\sigma_j)$ of $(\sigma_1,\ldots,\sigma_{k-1})$.
If $\sigma_{j+1}< k$, then, by the maximality of $\lambda$, we have $\sigma_{j+1}<\sigma_j$.

We claim that $\sigma$
can be gridded in $\GGG$ by placing
the points of $\tau$ in the upper right cell,
the points of $\lambda$ in the lower left cell,
and any remaining points in the lower right cell.
Clearly, both $\tau$ and $\lambda$ avoid $\mathbf{21}$.
Suppose there was an occurrence of $\mathbf{213}$ in the lower right cell.
The point $k$ can't occur to the left of the $\mathbf{213}$, since that would form a $\mathbf{4213}$.
Thus $k-1$ occurs to the right of $k$, $\sigma_{j+1}< k$, and thus $\sigma_{j+1}<\sigma_j$.
But that is impossible, because the points $\sigma_j$, $\sigma_{j+1}$, $k$ and $k-1$ would form a $\mathbf{2143}$.

On the other hand, we claim that there is no gridding of $\pi=\mathbf{4213}$ in $\GGG$. From Proposition~\ref{propH}, we know that such a gridding would require at least one point in the lower left cell.
Moreover, only the first point of $\pi$ can occur in that cell without creating a $\mathbf{21}$ there.
But, if the lower left cell were to contain only the first point of $\pi$, then
its remaining points would form a $\mathbf{213}$ in the lower right cell.
\end{proof}

Let us now determine the structure of griddings in $\GGG$ of permutations in $\DDD$.
Given a gridding of a permutation in $\GGG$, we refer to the points in the lower left cell as the \emph{left points}.
We also make use of the concept of a top or left point \emph{splitting}
(part of) a tree in the lower right cell.
A top point splits (part of) a tree if it occurs to the right of its leftmost point and to the left of its rightmost point.
Analogously, a left point splits (part of) a tree if it occurs above its lowermost point and below its uppermost point.
Splitting is illustrated in the figures 
by pale 
stripes.

\begin{propO}\label{propAGrid}
  A gridding of a permutation in $\GGG$ 
  avoids $\mathbf{2143}$ if and only if
  \begin{bullets}
  \item[(a)] no left point splits a tree below its uppermost non-trunk vertex, and
  \item[(b)] if a left point splits a tree, then no top point splits its trunk.
  \end{bullets}
\end{propO}

\begin{proof}
  If a left point $\ell$ splits a tree below its uppermost non-trunk vertex, then $\ell$ and the root, tip and uppermost non-trunk vertex of the tree form a $\mathbf{2143}$.
  Similarly, if a left point $\ell$ splits a tree and a top point $t$ splits its trunk, then $\ell$, the root of the tree, $t$ and the tip of the tree form a $\mathbf{2143}$.

\begin{figure}[ht]
\vspace{6pt}
$$
\begin{tikzpicture}[scale=0.25]
\path [fill=blue!15] (0.5,1.8) rectangle (5.5,2.2);
\draw[thick] (3.5,4)--(2.25,1)--(4.75,3);
\plotperm[blue!50!black]{5}{2}
\plotpt[blue!50!black]{2.25}{1}
\plotpt[blue!50!black]{3.5}{4}
\plotpt[blue!50!black]{4.75}{3}
\draw[thick] (0.5,4.5)--(5.5,4.5);
\draw[thick] (1.5,0.5)--(1.5,5.5);
\node[] at (3,-0.7) {$\Gamma_1$};
\end{tikzpicture}
\qquad\quad\qquad
\begin{tikzpicture}[scale=0.25]
\path [fill=blue!15] (0.5,2.3) rectangle (5.5,2.7);
\path [fill=blue!15] (3.3,0.5) rectangle (3.7,5.5);
\draw[thick] (2.25,1.25)--(4.75,3.75);
\plotperm[blue!50!black]{5}{2.5}
\plotpt[blue!50!black]{2.25}{1.25}
\plotpt[blue!50!black]{3.5}{5}
\plotpt[blue!50!black]{4.75}{3.75}
\draw[thick] (0.5,4.5)--(5.5,4.5);
\draw[thick] (1.5,0.5)--(1.5,5.5);
\node[] at (3,-0.7) {$\Gamma_2$};
\end{tikzpicture}
\vspace{-9pt}
$$
\caption{The two griddings of $\mathbf{2143}$ in $\GGG$}
\label{fig2143Gridding}
\end{figure}

  On the other hand, 
  there are only two griddings of $\mathbf{2143}$ in $\GGG$, which we denote $\Gamma_1$ and $\Gamma_2$ as shown in Figure~\ref{fig2143Gridding}.
  A gridding of a permutation in $\GGG$ can only contain an occurrence of $\Gamma_1$ if the left point of $\Gamma_1$ splits a tree below its uppermost non-trunk vertex.
  If a gridding of a permutation in $\GGG$ contains an occurrence of $\Gamma_2$ but no occurrence of $\Gamma_1$, then the left point of $\Gamma_2$ splits a single edge in the trunk of a tree and the top point of $\Gamma_2$ splits the same edge.
\end{proof}

We are now in a position to enumerate $\DDD$.

\begin{thmO}\label{thmAgf}
The generating function for $\DDD=\classD$ is
$$
D(z) \;=\;
\frac{(1-2\+z) \Big({-}1+5\+z-7\+z^2+2\+z^3 +(1-z) \sqrt{1-6\+z+5\+z^2}\Big)}{1-10\+z+24\+z^2-20\+z^3+4\+z^4} .
$$
\end{thmO}

\begin{figure}[ht]
$$
\begin{tikzpicture}[scale=0.25]
\path [fill=blue!15] (4.8,0.5) rectangle (5.2,22.5);
\path [fill=blue!15] (7.8,0.5) rectangle (8.2,22.5);
\path [fill=blue!15] (11.8,0.5) rectangle (12.2,22.5);
\path [fill=blue!15] (12.8,0.5) rectangle (13.2,22.5);
\path [fill=blue!15] (20.8,0.5) rectangle (21.2,22.5);
\path [fill=blue!15] (0.5,3.8) rectangle (22.5,4.2);
\path [fill=blue!15] (0.5,13.8) rectangle (22.5,14.2);
\draw[thick] (20,5)--(18,1)--(22,2);
\draw[thick] (10,6)--(11,7)--(14,11);
\draw[thick] (11,7)--(15,8)--(16,10);
\draw[thick] (15,8)--(17,9);
\draw[thick] (7,15)--(6,12)--(9,13);
\plotperm[blue!50!black]{22}{4,14,17,16,18,12,15,19,13,6,7,20,21,11,8,10,9,1,3,5,22,2}
\draw[thick] (0.5,17.5)--(22.5,17.5);
\draw[thick] (2.5,0.5)--(2.5,22.5);
\draw[thick,dashed] (0.5,16.5)--( 3.5,16.5)--( 3.5,22.5);
\draw[thick,dashed] (0.5,15.5)--( 5.5,15.5)--( 5.5,22.5);
\draw[thick,dashed] (0.5,11.5)--( 9.5,11.5)--( 9.5,22.5);
\draw[thick,dashed] (0.5, 5.5)--(17.5, 5.5)--(17.5,22.5);
\end{tikzpicture}
\vspace{-6pt}
$$
\caption{The canonical gridding in $\GGG$ of a permutation in $\DDD$}
\label{figDExample}
\end{figure}

\begin{proof}
For each permutation $\sigma$ in $\DDD$, we
consider the gridding of $\sigma$ in $\GGG$
in which the number of
left points
is minimal, with the remaining points
canonically gridded (as a gridding in $\HHH$) in the two cells at the right.
See Figure~\ref{figDExample} for an illustration.
We call this the \emph{canonical} gridding of $\sigma$ in $\GGG$.
Clearly, there is a bijection between permutations in $\DDD$ and their canonical griddings in $\GGG$.

\begin{figure}[ht]
\vspace{6pt}
$$
\begin{tikzpicture}[scale=0.25]
\path [fill=blue!15] (0.5,1.8) rectangle (5.5,2.2);
\draw[thick] (3.5,1)--(4.75,3);
\plotperm[blue!50!black]{5}{2}
\plotpt[blue!50!black]{2.25}{4}
\plotpt[blue!50!black]{3.5}{1}
\plotpt[blue!50!black]{4.75}{3}
\draw[thick] (0.5,4.5)--(5.5,4.5);
\draw[thick] (1.5,0.5)--(1.5,5.5);
\end{tikzpicture}
\vspace{-9pt}
$$
\caption{The canonical gridding of $\mathbf{2413}$ in $\GGG$}
\label{fig2413Gridding}
\end{figure}

The minimality of the set of left points implies that the uppermost left point is the first point in an occurrence of $\mathbf{2413}$.
Since the leftmost top point occurs to the right of a $\mathbf{21}$ in the lower left cell, this $\mathbf{2413}$ must be gridded as in Figure~\ref{fig2413Gridding}. This is the case if and only if the uppermost left point splits a tree and this tree is not the upper tree.

To enumerate $\DDD$, it thus suffices to enumerate griddings of permutations in $\GGG$ satisfying the following four constraints, the first two of which characterise the avoidance of $\mathbf{2143}$ (Proposition~\ref{propAGrid}), and the last two of which characterise canonical griddings:
  \begin{bullets}
  \item[(a)] No left point splits a tree below its uppermost non-trunk vertex.
  \item[(b)] If a left point splits a tree, then no top point splits its trunk.
  \item[(c)] The uppermost left point splits a tree and this tree is not the upper tree.
  \item[(d)] The leftmost top point occurs to the right of the (tip of the) trunk of the upper tree and not adjacent to it.
  \end{bullets}

Let us consider the set of permutations in $\DDD$ containing $\mathbf{2413}$ (i.e.~those whose griddings have at least one left point).

As in the proof of Proposition~\ref{propHgf}, we use the symbolic method to construct this set, reusing the definitions of $\PPP$, $\TsupT$ and $\Ttop$ in~\eqref{eqTpath}--\eqref{eqTtop} in the proof of Proposition~\ref{propHgf}.
We mark left points with~$\ell$ and top points with $t$ for clarity (although, again, we don't actually need any extra catalytic variables for the enumeration).
In Figure~\ref{figDExample}, the dashed lines illustrate how top and left points are associated with the trees.

The structural equations are as follows:
\begin{align}
  \Pleft     & \;=\;  \ZZZ \times \seqbig{\ZZZ \times \seq{\ell\+\ZZZ}}  \label{eqPleft}  \\[5pt]
  \Psplit & \;=\;  (\Pleft-\PPP)\times \seq{t\+\ZZZ}            \label{eqPleftr} \\[5pt]
  \Tsplit & \;=\;  \Psplit \:\;+\;\: \Psplit \times \seqplus{\TsupT} \times \seqbig{\ZZZ \times \seq{\TsupT}}  \label{eqTleftr} \\[5pt]
  \DDD \setminus \HHH & \;=\;  \Ttop \:\times\: \seq{\TsupT} \:\times\: \Tsplit \:\times\: \seqbig{\seq{\ell\+\ZZZ} \times (\TsupT + \Tsplit)}  \:\times\: \seq{\ell\+\ZZZ}  \label{eqD-H}
\end{align}
In~\eqref{eqPleft},
we use $\Pleft$ to denote a path tree 
together with any left points that split it.

In~\eqref{eqPleftr},
we use $\Psplit$ to denote a path tree split by at least one left point, together with any top points that occur to its right. By constraint~(b), a path tree may not be split by both left and top points.

In~\eqref{eqTleftr},
we use $\Tsplit$ to denote a tree split by at least one left point, together with any top points that occur to the right of some of its vertices. By constraint~(b), these top points must occur to the right of its tip.

An element of $\Tsplit$ is either a path tree split by at least one left point, or else is constructed (tip downwards) from 
      \begin{bullets}
        \item[(i)]   the trunk above the uppermost branching point, split by at least one left point, which may have top points to the right of the tip,
        \item[(ii)]  at least one subtree rooted at the the uppermost branching point on the trunk, which may have top points to the right of some of its vertices, and
        \item[(iii)] any points on the trunk below the uppermost branching point, each possibly rooting a subtree that may have top points to the right of some of its vertices.
      \end{bullets}

Finally, in~\eqref{eqD-H}, the set
of canonical griddings of permutations in $\DDD$ whose griddings have at least one left point
is constructed, starting at the upper left, from
      \begin{bullets}
        \item[(i)]   the upper tree, which, by constraint~(c), is never split by a left point,
        \item[(ii)]  zero or more $\TsupT$ trees, which are not split by a left point,
        \item[(iii)] a $\Tsplit$ tree, split by the uppermost left point,
        \item[(iv)]  zero or more $\TsupT$ or $\Tsplit$ trees, each possibly with left points above it, and
        \item[(v)]   zero or more left points below the lowermost tree.
      \end{bullets}
Translation into functional equations yields:
\begin{align*}
  \pleft(z,t,\ell)   & \;=\;  \seqfrac[z]{\seqfrac[z]{\ell\+z}} ,
  \\[6pt]
  \psplit(z,t,\ell) & \;=\;  \Big(\!\pleft(z,t,\ell) - \seqplusfrac{z}\Big) \seqfrac{t\+z} ,
  \\[6pt]
  \tsplit(z,t,\ell) & \;=\;  \psplit(z,t,\ell) \:\;+\;\: \seqplusfrac[\psplit(z,t,\ell)\+]{\tsupT(z,t)} \times \seqfrac{\seqfrac[z]{\tsupT(z,t)}} ,
  \\[6pt]
  D(z,t,\ell) - H(z,t)     & \;=\;  \seqfrac[\ttop(z,t)]{\tsupT(z,t)} \times
    \seqfrac[\tsplit(z,t,\ell)]{\seqfrac{\ell\+z}\big(\tsupT(z,t)+\tsplit(z,t,\ell)\big)}
    \times \seqfrac{\ell\+z} .
\end{align*}
Solving, expanding, simplifying and rearranging yields a (very complicated) algebraic expression for $D(z,t,\ell)$.
The univariate generating function then results from setting $t=1$ and $\ell=1$ and simplifying.
\end{proof}

From the generating function, we can determine the asymptotic behaviour.
\begin{corO}\label{corD}
The following properties hold in the limit:
\begin{bullets}
\item[(i)] Asymptotically, $|\DDD_n| \sim \frac{60}{121}\sqrt{\frac{5}{\pi}}\+5^n\+n^{-3/2}$, so
  the probability that a random large permutation avoiding $\mathbf{4213}$ and $\mathbf{2143}$
  also avoids $\mathbf{2413}$ tends to $\frac{121}{216} > \half$.
\item[(ii)] The number of left points in the canonical gridding of a permutation in $\DDD$ has asymptotic mean $\frac{175}{132}\approx 1.326$
  and standard deviation $\frac{\sqrt{74795}}{132}\approx2.072$.
\item[(iii)] The number of top points in the canonical gridding of a permutation in $\DDD_n$
is normally distributed with mean $\frac{n}{5}$ and standard deviation $\frac{2\sqrt{n}}{5}$.
\end{bullets}
\end{corO}

\begin{proof}
For part~(i), the asymptotics of $|\DDD_n|$ results from the application of Proposition~\ref{propAsympt} to the first nonconstant term of
the Puiseux expansion for $D(z)$ around its dominant singularity $\rho=\frac{1}{5}$:
$$
D(z) \;\sim\; \tfrac{9}{11} \:-\: \tfrac{129\sqrt{5}}{121}\sqrt{1-5\+z} \:+\: O(1-5\+z).
$$
Note that, although the denominator of the expression for $D(z)$ has a real root 
whose value is 
$\tfrac{1}{4} \big(5+\sqrt{5}-\sqrt{22+10 \sqrt{5}}\big) \approx
0.143922<1/5$, this value is also a root of the numerator.

The ratio $|\HHH_n|/|\DDD_n|$ gives
the proportion of permutations in $\DDD$ avoiding $\mathbf{2413}$. The asymptotic limit of this ratio is $\frac{121}{216}$.

For part~(ii), substituting $t=1$ and simplifying yields
$$
D(z,1,\ell) \;=\;
\frac{(1-z-z\+\ell)\+\Big(P_1(z,\ell) \:-\: (1 - z\+\ell)\+(1 - \ell + z\+\ell + z\+\ell^2) \sqrt{1-6\+z+5\+z^2} \Big)}{2\+z\+P_2(z,\ell)},
$$
where
\begin{align*}
P_1(z,\ell) & \;=\;
            (1 -    \ell)
\:-\: z  \+ (3 - 3\+\ell - 2\+\ell^2)
\:+\: z^2\+ (2 - 4\+\ell - 7\+\ell^2 -    \ell^3) \\ & \qquad\qquad\qquad\qquad\qquad\qquad\qquad\qquad
\:+\: z^3\+ (    2\+\ell + 9\+\ell^2 + 3\+\ell^3)
\:-\: z^4\+ (              2\+\ell^2 + 2\+\ell^3) , \\[5pt]
P_2(z,\ell) & \;=\;
            (2 - 3\+\ell)
\:-\: z  \+ (3 - 5\+\ell -  8\+\ell^2)
\:+\: z^2\+ (1 - 3\+\ell - 15\+\ell^2 - 7\+\ell^3) \\ & \qquad\qquad\qquad\qquad\qquad\qquad\qquad\qquad
\:+\: z^3\+ (       \ell +  8\+\ell^2 + 9\+\ell^3 + 2\+\ell^4)
\:-\: z^4\+ (                  \ell^2 + 2\+\ell^3 +    \ell^4) .
\end{align*}
The asymptotic mean and standard deviation for the number of left points follow by the application of Propositions~\ref{propAsympt} and~\ref{propMoments}. The limiting distribution is illustrated in Figure~\ref{figLeftPointDistrib}.
\begin{figure}[t]
  \begin{center}
    \includegraphics[scale=0.6]{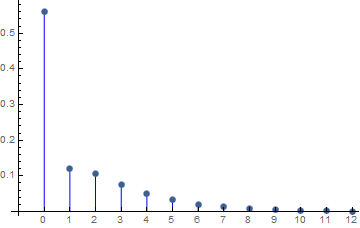}
    \vspace{-12pt}
  \end{center}
  \caption{The asymptotic distribution of the number of left points in the canonical gridding of a permutation in class $\DDD$}\label{figLeftPointDistrib}
\end{figure}

For part~(iii),
, substituting $\ell=1$ and simplifying yields
$$
D(z,t,1) \;=\;
\frac{(1 - 2\+z)\+\Big(P_3(z,t) \:-\: (1 - z)\+(1 + t - 2\+t\+z) \sqrt{1 - (2\+t+4)\+z + t\+(t+4)\+z^2}\Big)}{2\+P_4(z,t)},
$$
where
\begin{align*}
P_3(z,t) & \;=\; (1 - 2\+z)\+(1 - t\+z)\+\big((1 + t)\+(1 - 3\+z) + 2\+t\+z^2\big) , \\[5pt]
P_4(z,t) & \;=\;
      - 1
\:+\: ( 7 +  3\+t +     t^2)\+z
\:-\: (14 + 14\+t +  6\+t^2)\+z^2
\:+\: ( 9 + 22\+t + 13\+t^2)\+z^3 \\ & \qquad\qquad\qquad\qquad\qquad\qquad\qquad\qquad\qquad\qquad\qquad
\:-\: (     12\+t + 12\+t^2)\+z^4
\:+\:                4\+t^2 \+z^5 .
\end{align*}
The asymptotic mean and standard deviation for the number of left points follow by the application of Propositions~\ref{propAsympt} and~\ref{propMoments}.
 
Although $D(z,t,1)$ can easily be expressed in the form $A(z,t)+B(z,t)\sqrt{C(z,t)}$, when this is done, both
$A(z,1)$ and $B(z,1)$ have singularities at
$z 
\approx 0.143922 < \tfrac{1}{5}$.
Thus, Proposition~\ref{propGaussian} is not applicable and can't be used to deduce
an asymptotic Gaussian distribution.

However, it is not difficult to rearrange the functional equations to show that
$D=D(z,t,1)$
is a solution of the following system of polynomial equations,
\begin{align*}
R_1 &\;=\; 1 \:+\: z\+R_1                                           \\
R_2 &\;=\; 1 \:+\: z\+t\+R_2                                        \\
T   &\;=\; z\+R_2 \:+\: T^2                                         \\
R_3 &\;=\; 1 \:+\: T\+R_3                                           \\
R_4 &\;=\; 1 \:+\: z\+R_3\+R_4                                      \\
S   &\;=\; z^3\+R_1^2\+R_2\+(1 + T\+R_3\+R_4) \:+\: z\+R_1\+S       \\
R_5 &\;=\; 1 \:+\: R_1\+(T + S)\+R_5                                \\
D   &\;=\; z\+R_1\+R_3\+(1 + z\+T\+R_3\+R_4)\+(1 + R_1\+S\+R_5)     ,
\end{align*}
in which $T=T(z,t)$ and $S=S(z,t,1)$.
The significance
of this particular reformulation is that each polynomial has positive coefficients.
This allows us to apply a variant of the limiting distribution version of the 
Drmota--Lalley--Woods Theorem~\cite{Drmota1997}, which states that, under very general conditions,
a parameter of a combinatorial class satisfying such a system of equations exhibits a Gaussian limit law 
(see also~\cite[Proposition~IX.17 and subsequent discussion]{FS2009}).

The original version of this theorem requires the dependency graph of the system of equations to be strongly connected, 
something that is evidently not the case here.
However, a recent generalisation by 
Banderier \& Drmota~\cite[Theorem~8.2]{BD2015}
addresses the non-strongly connected scenario, and, with care, can be applied in this case 
to prove
that the number of top points does indeed 
exhibit a Gaussian limit distribution concentrated around $n/5$.
We leave the details to be filled in by the assiduous reader.
\end{proof}

The first twelve terms of the sequence $|\DDD_n|$ are 1, 2, 6, 22, 88, 366, 1556, 6720, 29396, 129996, 580276, 2611290.
More values
can be found at
\href{http://oeis.org/A165537}{A165537} in OEIS~\cite{OEIS}.

\section{Questions}\label{sectQ}

The somewhat unexpected result that more than half of the permutations in $\classD$ are, in fact, members of $\classH$,
is a consequence of the fact that
$|\HHH_n|/|\DDD_n|$ has a positive limit as $n$ tends to infinity.
This phenomenon
deserves further investigation, as does the weaker property that the growth rates of $\HHH$ and $\DDD$ are equal: $\gr(\HHH)=\gr(\DDD)$.
\begin{questO}
When is $\liminfty \dfrac{\big|\av_n(B\cup\{\beta\})\big|}{\big|\av_n(B)\big|}$ positive?
\end{questO}
\begin{questO}
When do we have $\gr\big(\av(B\cup\{\beta\})\big) = \gr\big(\av_n(B)\big)$?
\end{questO}

We conclude with the observation which initiated this investigation.
Suppose $\KKK=\{\SSS_1,\ldots,\SSS_k\}$ is a finite collection of permutation classes.
We say that another permutation class $\CCC$ is \emph{$\KKK$-griddable} if there is a finite matrix $M$
each entry of which is one of the $\SSS_i$,
such that $\CCC$ is a subset of the generalised grid class $\Grid(M)$. 
Vatter gave the following characterisation of the property of being \mbox{$\KKK$-griddable}.
\begin{propO}[{Vatter~\cite[Theorem~3.1]{Vatter2011}}]\label{propKGriddable}
The permutation class $\CCC$ is $\KKK$-griddable if and only if it
does not contain arbitrarily long sums or skew sums of basis elements
of members of $\KKK$, that is, if there exists a constant $m$ so that $\CCC$ does not contain
$\beta_1\oplus\ldots\oplus\beta_m$ or $\beta_1\ominus\ldots\ominus\beta_m$ for any sequence $\beta_1,\ldots,\beta_m$ of basis
elements of members of $\KKK$.
\end{propO}
Consequently,
three of the permutation classes
avoiding two patterns of length four
that were
still waiting to be
enumerated when this paper was first written (see~\cite{WikiEnumPermClassesThin})
are subclasses of a generalised grid class.
\begin{bullets}
\item $\av(\mathbf{4312}, \mathbf{3214})$ is $\{\av(\mathbf{12}) , \av(\mathbf{321})\}\+$-griddable.
\item $\av(\mathbf{4231}, \mathbf{3214})$ is $\{\av(\mathbf{12}) , \av(\mathbf{321})\}\+$-griddable.
\item $\av(\mathbf{4132}, \mathbf{3214})$ is $\{\av(\mathbf{132}) , \av(\mathbf{321})\}\+$-griddable.
\end{bullets}
Miner~\cite{Miner2016} has subsequently derived the generating functions for these classes, using techniques very similar to those used here.

\acknowledgements
The author is grateful to Cyril Banderier for pointing him in the right direction concerning proving the Gaussian limit law for top points in $\DDD$.
Michael Albert's \emph{PermLab} software~\cite{PermLab} was helpful for visualising the structure of permutations in class $\DDD$.
\emph{Mathematica}~\cite{Mathematica} was used for algebraic manipulation.

\emph{Soli Deo gloria!}

\bibliographystyle{plain}
\small
\bibliography{../bib/mybib}

\begin{thebibliography}{10}

\bibitem{PermLab}
Michael Albert.
\newblock Perm{L}ab: Software for permutation patterns.
\newblock \href{http://www.cs.otago.ac.nz/PermLab}{www.cs.otago.ac.nz/PermLab},
  2012.

\bibitem{Atkinson1999}
M.~D. Atkinson.
\newblock Restricted permutations.
\newblock {\em Discrete Math.}, 195:27--38, 1999.

\bibitem{BD2015}
Cyril Banderier and Michael Drmota.
\newblock Formulae and asymptotics for coefficients of algebraic functions.
\newblock {\em Combin. Probab. Comput.}, 24(1):1--53, 2015.

\bibitem{BevanPPBasics}
David Bevan.
\newblock Permutation patterns: basic definitions and notation.
\newblock {\em \href{http://arxiv.org/pdf/1506.06673}{arXiv:1506.06673}}, 2015.

\bibitem{BevanAv1324GR}
David Bevan.
\newblock Permutations avoiding 1324 and patterns in {{\L}}ukasiewicz paths.
\newblock {\em J.~London Math. Soc.}, 92(1):105--122, 2015.

\bibitem{BevanTwoClasses}
David Bevan.
\newblock The permutation classes {A}v(1234,2341) and {A}v(1243,2314).
\newblock {\em Australasian J. Combin.}, 64(1):3--20, 2016.

\bibitem{BB2014}
Jonathan Bloom and Alexander Burstein.
\newblock Egge triples and unbalanced {W}ilf-equivalence.
\newblock {\em Australasian J. Combin.}, 64(1):232--251, 2016.

\bibitem{Bona2012}
Mikl{\'o}s B{\'o}na.
\newblock {\em Combinatorics of Permutations}.
\newblock Discrete Mathematics and its Applications. CRC Press, second edition,
  2012.

\bibitem{BMB2007}
Mireille Bousquet-M{\'e}lou and Steve Butler.
\newblock Forest-like permutations.
\newblock {\em Ann. Comb.}, 11:335--354, 2007.

\bibitem{Drmota1997}
Michael Drmota.
\newblock Systems of functional equations.
\newblock {\em Random Structures Algorithms}, 10(1--2):103--124, 1997.

\bibitem{FO1990}
Philippe Flajolet and Andrew Odlyzko.
\newblock Singularity analysis of generating functions.
\newblock {\em SIAM J. Discrete Math.}, 3(2):216--240, 1990.

\bibitem{FS2009}
Philippe Flajolet and Robert Sedgewick.
\newblock {\em Analytic Combinatorics}.
\newblock Cambridge University Press, 2009.

\bibitem{Hwang1998}
H.-K. Hwang.
\newblock On convergence rates in the central limit theorems for combinatorial
  structures.
\newblock {\em European J. Combin.}, 19(3):329--343, 1998.

\bibitem{Kitaev2011}
Sergey Kitaev.
\newblock {\em Patterns in Permutations and Words}.
\newblock Springer, 2011.

\bibitem{Miner2016}
Sam Miner.
\newblock Enumeration of several two-by-four classes.
\newblock {\em Preprint.
  \href{http://arxiv.org/pdf/1610.01908}{arXiv:1610.01908}}, 2016.

\bibitem{OEIS}
The OEIS Foundation~Inc.
\newblock The {O}n-{L}ine {E}ncyclopedia of {I}nteger {S}equences.
\newblock Published electronically at
  \href{https://oeis.org}{https:/$\!$/oeis.org}.

\bibitem{Vatter2011}
Vincent Vatter.
\newblock Small permutation classes.
\newblock {\em Proc. London Math. Soc.}, 103(5):879--921, 2011.

\bibitem{VatterSurvey}
Vincent Vatter.
\newblock Permutation classes.
\newblock In Mikl{\'o}s B{\'o}na, editor, {\em The Handbook of Enumerative
  Combinatorics}. CRC Press, 2015.

\bibitem{WikiEnumPermClassesThin}
Wikipedia.
\newblock Enumerations of specific permutation classes.
\newblock
  \href{http://en.wikipedia.org/wiki/Enumerations\_of\_specific\_permutation\_classes}{http:/$\!$/en.wikipedia.org/wiki/Enumera}
  \href{http://en.wikipedia.org/wiki/Enumerations\_of\_specific\_permutation\_classes}{tions\_of\_specific\_permutation\_classes}.

\bibitem{Mathematica}
{Wolfram Research, Inc.}
\newblock Mathematica. {V}ersion 10.0.
\newblock
  \href{http://www.wolfram.com/mathematica}{www.wolfram.com/mathematica}, 2014.

\end{thebibliography}

\end{document}